\documentclass[a4paper, 12pt]{article}

\usepackage{amsmath, amsthm, amssymb}
\usepackage{graphicx}
\usepackage{enumitem}
\usepackage{authblk}
\usepackage[textsize=small, textwidth=2cm, color=yellow]{todonotes}
\usepackage{thmtools}
\usepackage{thm-restate}

\usepackage[utf8]{inputenc}
\usepackage[russian, english]{babel}
\usepackage[T1]{fontenc}
\usepackage[a4paper, margin=2.0cm]{geometry}
\usepackage{needspace}

\usepackage{libertine} 
\usepackage{inconsolata} 

\newcommand{\defn}[1]{\textcolor{violet}{\emph{#1}}}

\usepackage[hyphens]{url}
\usepackage[linktoc = all, hidelinks, colorlinks, unicode=true, pagebackref=true]{hyperref}
\usepackage[capitalise, compress, nameinlink, noabbrev]{cleveref}
\hypersetup{linkcolor={black}, citecolor={blue!70!black}, urlcolor={blue!70!black}}

\declaretheorem[name=Theorem, numberwithin=section]{theorem}
\declaretheorem[name=Lemma, sibling=theorem]{lemma}
\declaretheorem[name=Proposition, sibling=theorem]{proposition}

\declaretheorem[name=Conjecture, sibling=theorem]{conjecture}

\def\cqedsymbol{\ifmmode$\lrcorner$\else{\unskip\nobreak\hfil
		\penalty50\hskip1em\null\nobreak\hfil$\lrcorner$
		\parfillskip=0pt\finalhyphendemerits=0\endgraf}\fi} 

\newcommand{\vphi}{\varphi}

\let\le\leqslant
\let\ge\geqslant
\let\leq\leqslant
\let\geq\geqslant

\title{On high genus extensions of Negami's conjecture}
\usepackage[symbol]{footmisc}

\stepcounter{footnote}

\author[1]{Marcin Bria\'nski \thanks{Partially supported by Polish National Science Centre grant number 2019/34/E/ST6/00443.}}
\author[2]{James Davies}
\author[3]{Jane Tan}

\affil[1]{Jagiellonian University, Krak\'ow, Poland} 
\affil[2]{University of Cambridge, United Kingdom}
\affil[3]{University of Oxford, United Kingdom.}
\date{}

\begin{document}
	
	\maketitle
	
	\begin{abstract}
		Negami's famous planar cover conjecture is equivalent to the statement that a connected graph can be embedded in the projective plane if and only if it has a projective planar cover. In 1999, Hlin{\v{e}}n{\`y} proposed extending this conjecture to higher genus non-orientable surfaces. In this paper, we put forward a natural extension that encompasses orientable surfaces as well; for every compact surface $\Sigma$, a connected graph \(G\) has a finite cover embeddable in $\Sigma$ if and only if \(G\) is embeddable in a surface covered by $\Sigma$.
		
		As evidence toward this, we prove that for every surface $\Sigma$, the connected graphs with a finite cover embeddable in $\Sigma$ have bounded Euler genus. 
		Moreover, we show that these extensions of Negami's conjecture are decidable for every compact surface of sufficiently large Euler genus, surpassing what is known for Negami's original conjecture.
		We also prove the natural analogue for countable graphs embeddable into a compact (orientable) surface.
		More precisely, we prove that a connected countable graph \(G\) has a finite ply cover that embeds into a compact (orientable) surface if and only if \(G\) embeds into a compact (orientable) surface.
		
		Our most general theorem, from which these results are derived, is that there is a constant $c>0$ such that for every surface $\Sigma$, there exists a decreasing function $p_\Sigma:\mathbb{N} \to \mathbb{N}$ with $\lim_{g\to \infty}p_\Sigma(g) =0$ such that every finite cover embeddable in $\Sigma$ of any connected graph with Euler genus $g\ge c$ has ply at most $p_\Sigma(g)$.
	\end{abstract}
	
	\section{Introduction}
	
	A graph $\widehat{G}$ is a \defn{cover} of a graph $G$ if there exists a surjective map $\varphi: V(\widehat{G}) \to V(G)$ such that, for every $\widehat{v}\in V(\widehat{G})$, the neighbours of $\widehat{v}$ are mapped bijectively to the neighbours of $\vphi(\widehat{v})$. Given a subgraph $S\subset G$, the \defn{lift} $\widehat{S}$ of $S$ (into $\widehat{G}$) is the preimage $\vphi^{-1}$. When $G$ is connected and $\widehat{G}$ is finite, there is a natural number $p$ such that $|\vphi^{-1}(v)|=p$ for every $v\in V(G)$. We call this common preimage size $p$ the \defn{ply} of the cover. If $\widehat{G}$ is planar, then we say that $G$ has a \defn{planar cover}. 
	
	In 1988, Negami \cite{negami1988spherical} made the beautiful conjecture that a connected graph has a finite planar cover if and only if it can be embedded on the projective plane.
	
	\begin{conjecture}[Negami \cite{negami1988spherical}, 1988]\label{conj:Negami}
		A connected graph has a finite planar cover if and only if it is projective-planar.
	\end{conjecture}
	
	One direction is easy: since the sphere is a double cover of the projective plane, every projective-planar graph has a finite planar cover. For the converse, the main approach has been based on the observation that if a graph $H$ has a finite planar cover, then so does every minor of $H$. Building on work of Glover, Huneke and Wang \cite{GHW79}, Archdeacon \cite{archdeacon1981kuratowski} proved a Kuratowski theorem for projective-planar graphs in 1981, characterizing embeddability on the projective plane by a list of 35 forbidden minors. To prove Negami's planar cover conjecture, it would suffice to show that none of the connected forbidden minors in this list have a planar cover. By 1998, the combined work of Archdeacon \cite{archdeacon2002two}, Hlin{\v{e}}n{\`y} \cite{hlineny1998k4}, Fellows \cite{fellows}, and Negami \cite{negami1988} had settled all but one case, reducing Negami's conjecture to whether $K_{1,2,2,2}$ has a finite planar cover. For an excellent survey on this progress, we refer to \cite{hlineny201020}. Notably, this last case has remained open even after another 25 years. 
	
	This is not to say that there have been no further advancements. Most relevantly to the present paper, Annor, Nikolayevsky and Payne \cite{annor2023counterexamples} recently proved that $K_{1,2,2,2}$ has no finite planar cover with ply at most 13. This means that any cover in a counterexample to Negami's conjecture must have ply at least 14. However, there is no known upper bound on the ply of a smallest finite planar cover of $K_{1,2,2,2}$, and consequently no upper bound on the ply of a counterexample. Such a result would be very attractive as it would render the problem finite.
	
	In 1999, Hlin{\v{e}}n{\`y} \cite{hlineny1999note} considered extensions of Negami's conjecture to higher genus surfaces. The surfaces in this paper are connected compact 2-manifolds without boundary, and hence specified by their (non)-orientability and genus as per the classification of surfaces (see \cite[Chapter 3.1]{mohar2001graphs}). Generalising planar covers, if a graph $G$ has a cover $\widehat{G}$ that is embeddable in a surface $\Sigma$, then we call this a \defn{$\Sigma$-cover}. An equivalent statement to \cref{conj:Negami} is then that a connected graph can be embedded in the projective plane if and only if it has a projective cover.
	Hlin{\v{e}}n{\`y} \cite{hlineny1999note} gave this formulation and proposed the following generalisation, where the projective plane is replaced by non-orientable surfaces of any genus.
	
	\begin{conjecture}[Hlin{\v{e}}n{\`y} \cite{hlineny1999note,hlineny1999planar}, 1999]\label{conj:Hliney}
		Let $\Sigma$ be a non-orientable surface.
		Then, a connected graph has a finite $\Sigma$-cover if and only if it is embeddable in $\Sigma$.
	\end{conjecture}
	
	The preceding statement cannot be true for orientable surfaces since there are graphs with planar covers that have arbitrarily high orientable genus \cite{ABY63}. Nonetheless, the fact that the orientable surface of Euler genus $g$ covers itself and the non-orientable surface of Euler genus $g/2 + 1$ (see, for instance, \cite{hatcher2005algebraic}) leads to a natural analogue for orientable surfaces.
	
	\begin{conjecture}\label{con:main}
		For every orientable surface $\Sigma$ of Euler genus $g$, a connected graph has a finite $\Sigma$-cover if and only if it is embeddable in $\Sigma$ or embeddable in the non-orientable surface of genus $g/2 + 1$.
	\end{conjecture}
	
	We also propose the following conjecture that unifies \cref{conj:Hliney} and \cref{con:main}.
	
	\begin{conjecture}\label{con:main2}
		For every surface $\Sigma$ of Euler genus $g$,
		a connected graph has a finite $\Sigma$-cover if and only if it is embeddable in a surface finitely covered by $\Sigma$.
	\end{conjecture}
	
	In this paper, we provide evidence toward \cref{con:main2} and a theoretical framework for verifying it. Our main theorem, from which will derive several key consequences, does what has not yet been achieved for Negami's conjecture and provides an upper bound on the ply of a $\Sigma$-cover for graphs with sufficiently high Euler genus.
	
	\begin{theorem}
		\label{thm:finply2}
		There exists a constant\footnote{We do not obtain an explicit value for the constant $c$ given in \cref{thm:finply2} since our proof relies on a theorem of Robertson and Seymour \cite{robertson24b} (see \cref{thm:RS}) which does not give explicit bounds. This constant is computable, however, since the bounds used from \cref{thm:RS} can also be obtained by computing forbidden minors for surfaces with small Euler genus.} $c>0$ such that the following holds. If $G$ is a connected graph with Euler genus $g\ge c$, then every finite cover of $G$ that is embeddable in a surface $\Sigma$ has ply at most $p_\Sigma(g)$, where $p_\Sigma:\mathbb{N} \to \mathbb{N}$ is a decreasing function depending only on $\Sigma$ with $\lim_{g\to \infty}p_\Sigma(g) =0$.
	\end{theorem}

	The first significant implication of \cref{thm:finply2} is that graphs with finite $\Sigma$-covers have bounded Euler genus. Fix a surface $\Sigma$, and choose $g$ sufficiently large so that $g \ge c$ and $p_\Sigma(g) < 1$. The latter is possible since $\lim_{g\to \infty}p_\Sigma(g) =0$.
	With these parameters, \cref{thm:finply2} implies that any connected graph $G$ with Euler genus at least $g$ does not have a finite $\Sigma$-cover.
	This was previously only known in the cases when $\Sigma$ is the sphere or projective plane \cite{hlineny2001another,hlineny2004possible}. 
	\begin{theorem}
		\label{thm:finplyEulergenus}
		For every surface $\Sigma$, the connected graphs with a finite $\Sigma$-cover have bounded Euler genus.
	\end{theorem}
	
	The second significant implication of \cref{thm:finply2} is that \cref{con:main2} is decidable for each surface of Euler genus at least $c$. To see this, let us fix such a surface $\Sigma$. A result of Seymour \cite{seymour1993bound} gives an explicit upper bound on the size of a minimal forbidden minor of a graph embeddable in any fixed surface. 
	Since there are algorithms for testing if a given graph is embeddable in a given surface (see \cite{mohar1999linear}), one can find in finite time the minimal forbidden minors for any surface covered by $\Sigma$.
	Given the forbidden minors for two classes of graphs (the classes of graphs embeddable in two surfaces covered by $\Sigma$ respectively, say), one can also find in finite time the minimal forbidden minors for the union of the two classes \cite{robertson2009graph}. Putting this together, we can find in finite time the (finite list of) minimal forbidden minors $\mathcal{H}_\Sigma$ that characterise the class of graphs embeddable in some surface covered by $\Sigma$.
	
	To prove (or disprove) \cref{con:main2} for $\Sigma$, it is then enough to consider the graphs in $\mathcal{H}_\Sigma$. 
	Suppose that all graphs in $\mathcal{H}_\Sigma$ have at most $n_\Sigma$ vertices.
	Then, one can enumerate all graphs with at most $n_\Sigma p_\Sigma$ vertices and test one by one whether they are embeddable in $\Sigma$ and whether they cover a graph in $\mathcal{H}_\Sigma$.
	If such a graph exists, then it (along with its embedding in $\Sigma$) provides a counter-example to \cref{con:main2} for the surface $\Sigma$.
	Otherwise, by \cref{thm:finply2},  if no such graph exists then this algorithm provides a proof of \cref{con:main2} for the surface $\Sigma$.
	Thus, we obtain:
	
	\begin{theorem}\label{thm:decidable}
		For each surface $\Sigma$ of Euler genus at least $c$, there is a finite time algorithm that determines whether or not \cref{con:main2} is true for the surface $\Sigma$, and produces a counter-example if it its not true.
	\end{theorem}
	
	The caveat here is that the algorithm described above will in general require a huge number of steps, so it does not directly give a practical answer to our question.
	It may be possible, however, to extend the ideas of the present paper to improve \cref{thm:finply2} and settle \cref{con:main2} for at least some surfaces. In addition, the preceding theorem is already in contrast to Negami's conjecture, where the question of decidability is still open. It is interesting that this result brings us closer in some sense to proving high genus analogues of Negami's conjecture than proving Negami's conjecture itself, despite the latter being reduced to essentially just one remaining case.
	One other famous open conjecture in graph minor theory which is known to be decidable (for each $k$) is Hadwiger's conjecture \cite{kawarabayashi2009hadwiger}.

	The proof of \cref{thm:finply2} crucially uses a structural theorem of Robertson and Seymour \cite{robertson24b} (see \cref{thm:RS}) concerning excluded minors, which similarly to Archdeacon's theorem \cite{archdeacon1981kuratowski} for projective-planar graphs allows us to reduce to studying covers of a very specific family of graphs. While we have so far worked with finite graphs, these tools all have counterparts for infinite (countable) graphs that are actually even stronger. Hence, the same outline proves an analogue of Negami's conjecture for countable graphs embeddable in a compact surface.
	For this, we use a recent forbidden minor theorem of Georgakopoulos \cite{georgakopoulos2023excluded} (see \cref{thm:Georgakopoulos}).
	
	\begin{theorem}
		\label{thm:finply_countable}
		A connected countable graph $G$ has a finite ply cover that embeds into a compact (orientable) surface if and only if $G$ embeds into a compact (orientable) surface.
	\end{theorem}
	
	Note that the surfaces in \cref{thm:finply_countable} can be restricted to orientable ones since every countable graph embeddable in a compact surface is embeddable in a compact orientable surface \cite{georgakopoulos2023excluded}.
	
	In \cref{sec:surfaceminors}, we discuss the aforementioned structural results for graphs on surfaces in both the finite and countable cases. At the end of that section, we state the remaining result connecting structure to ply that we need to prove \cref{thm:finply2} and \cref{thm:finply_countable} and give these arguments. The purpose of \cref{sec:sumkuratowski} is then fill in this last result, which entails verifying that all $\Sigma$-covers of the excluded graphs from our structural results have bounded ply. For basic concepts of topological graph theory we follow \cite{mohar2001graphs}, and for standard facts about surfaces we refer to \cite{hatcher2005algebraic}. 
	
	\section{Surfaces and forbidden minors}
	\label{sec:surfaceminors}

	When dealing with embedded graphs, a useful operation that preserves many topological properties is the $Y\Delta$-transformation, which is performed by deleting a vertex $v$ of degree 3 and adding three edges between the vertices of $N(v)$ to form a triangle. We will study a slightly more general relation that still works well with covers. Say that a graph $G$ is a Y-minor of a graph $H$ if $G$ can be obtained from $H$ by any sequence of the following operations: adding an edge between two neighbours of a degree 3 vertex, deleting a vertex, deleting an edge, and contracting an edge. This is slightly more general than the usual notion of minors (in particular, note that if $H$ is a minor of $G$, then $H$ is also a Y-minor of $G$).
	
	\begin{proposition}\label{prop:yminorcover}
		If a graph $H$ has a $p$-ply cover in a surface $\Sigma$, then so does every Y-minor of $H$.
	\end{proposition}
	
	\begin{proof}
		Suppose that $H$ has a $p$-ply cover $\widehat{H}$, via the map $\vphi: V(\widehat{H}) \to V(H)$, embedded in $\Sigma$. We show that each operation in the definition of a Y-minor preserves the property of having such a cover. This is clear for the graph minor operations using the fact that $\vphi$ is a bijection on neighbourhoods, and that the class of graphs embeddable in $\Sigma$ is minor-closed:
		\begin{itemize}
			\item If $G$ is obtained from $H$ by deleting any edge, then $\widehat{H}$ is a cover of $G$ via the same map $\vphi$.
			\item If $G$ is obtained from $H$ by deleting a vertex $v$, then $\widehat{H} \setminus \vphi^{-1}(v)$ is a cover of $G$ with a suitable restriction of $\vphi$.
			\item If $G$ is obtained from $H$ by contracting an edge $ab$ to obtain a new vertex $c$, then let $\widehat{G}$ be obtained from $\widehat{H}$ by contracting all edges of the form $\widehat{a}\widehat{b} \in E(\widehat{H})$ where $\widehat{a}\in \vphi^{-1}(a)$ and $\widehat{b} \in \vphi^{-1}(b)$. Then $\widehat{G}$ is a cover of $G$ with map $\vphi': V(\widehat{G}) \to V(G)$ define by $\vphi'(v) = c$ if $v$ is a vertex formed by contraction, and $\vphi'(v) = \vphi(v)$ otherwise.
		\end{itemize}
		
		For the last operation, let $G$ be a graph obtained from $H$ by adding one edge $ab$ between two neighbours of a degree 3 vertex $c$. To define a suitable cover, we take $\widehat{G}$ to be the graph obtained from $\widehat{H}$ by adding for each vertex $\widehat{v} \in \vphi^{-1}(v)$ an edge between the (unique) neighbours $\widehat{a}$ and $\widehat{b}$ of $\widehat{v}$ such that $\widehat{a} \in \vphi^{-1}(a)$ and $\widehat{b} \in \vphi^{-1}(b)$. For every such triple $\widehat{v}, \widehat{a}, \widehat{b} \in V(\widehat{H})$, since $\widehat{v}$ has degree 3 then $\widehat{a}$ and $\widehat{b}$ must appear consecutively in the cyclic ordering at $\widehat{v}$, there is a face of $\widehat{H}$ incident to both edges $\widehat{a}\widehat{v}$ and $\widehat{b}\widehat{v}$. Thus, we can draw $\widehat{a}\widehat{b}$ into this face, and it follows that $\widehat{G}$ is embeddable in $\Sigma$. In addition, viewing $\vphi$ as a map $V(\widehat{G}) \to V(G)$, it is clear by construction that neighbourhoods are mapped bijectively, so $\widehat{G}$ is a $p$-ply $\Sigma$-cover of $G$.
	\end{proof}
	
	\cref{prop:yminorcover} facilitates the strategy for \cref{thm:finply2} and \cref{thm:finply_countable}; we want to find a family such that every connected graph with high enough Euler genus contains a member of the family as a Y-minor, and such that we can prove these theorems for each member of the family.
	Kuratowski's \cite{kuratowski1930} famous theorem characterises planar graphs as precisely the graphs containing no subdivision of $K_5$ or $K_{3,3}$, known as the Kuratowski subgraphs, and Wagner \cite{wagner1937} proved an equivalent statement excluding the Kuratowski subgraphs as minors. Our goal in this section is to prove such a theorem for (finite and countable) connected graphs. 
	In the finite case, we shall derive this from a theorem of Robertson and Seymour \cite{robertson24b} (\cref{thm:RS}), 
	and in the countable case, we employ a recent theorem of Georgakopoulos \cite{georgakopoulos2023excluded} (\cref{thm:Georgakopoulos}).
	
	The excluded structures that we need to consider are built from sums of Kuratowski subgraphs as described below, with some examples depicted in \cref{fig:sum-kuratowski}. For $k\in \mathbb{N} \cup \{\infty\}$:
	\begin{itemize}
		\item let $\Lambda_{1,k}$ be the disjoint union of $k$ copies of $K_5$;
		\item let $\Lambda_{2,k}$ be the disjoint union of $k$ copies of $K_{3,3}$;
		\item let $\Omega_{1,k}$ be the graph obtained by taking $k$ copies of $K_5$ and identifying \textbf{one} vertex from each copy together;
		\item let $\Omega_{2,k}$ be the graph obtained by taking $k$ copies of $K_{3,3}$ and identifying \textbf{one} vertex from each copy together;
		\item let $\Theta_{1,k}$ be the graph obtained by taking $k$ copies of $K_5$ and identifying \textbf{two} vertices from each copy together;
		\item let $\Theta_{2,k}$ be the graph obtained by taking $k$ copies of $K_{3,3}$ and identifying \textbf{two adjacent} vertices from each copy together;
		\item let $\Theta_{3,k}$ be the graph obtained by taking $k$ copies of $K_{3,3}$ and identifying \textbf{two non-adjacent} vertices from each copy together;%
		\item let $\Pi_{1,k}$ be the graph obtained by taking $k$ copies of $K_5$, say $K_5^1, \ldots, K_5^k$, with chosen pairs of \textbf{distinct} vertices $v^i_a$ and $v^i_b$ in each $K^i_5$, and joining them along a path by identifying $v_a^i$ and $v_b^{i+1}$ for each $1\leq i < k$;
		\item let $\Pi_{2,k}$ be the graph obtained by taking $k$ copies of $K_{3,3}$, say $K_{3,3}^1, \ldots, K_{3,3}^k$, with chosen pairs of \textbf{adjacent} vertices $v^i_a$ and $v^i_b$ in each $K_{3,3}^i$, and joining them along a path by identifying $v_a^i$ and $v_b^{i+1}$ for each $1\leq i < k$;
		\item let $\Pi_{3,k}$ be the graph obtained by taking $k$ copies of $K_{3,3}$, say $K_{3,3}^1, \ldots, K_{3,3}^k$, with chosen pairs of \textbf{non-adjacent} vertices $v^i_a$ and $v^i_b$ in each $K_{3,3}^i$, and joining them along a path by identifying $v_a^i$ and $v_b^{i+1}$ for each $1\leq i < k$;
		\item let $K_{3,k}$ be the complete bipartite graph with 3 vertices in one part and $k$ in the other.
	\end{itemize}

	\renewcommand{\arraystretch}{1}
	\begin{figure}
		\centering
		\begin{tabular}{ccccc}
			\includegraphics[page=3,scale=0.12]{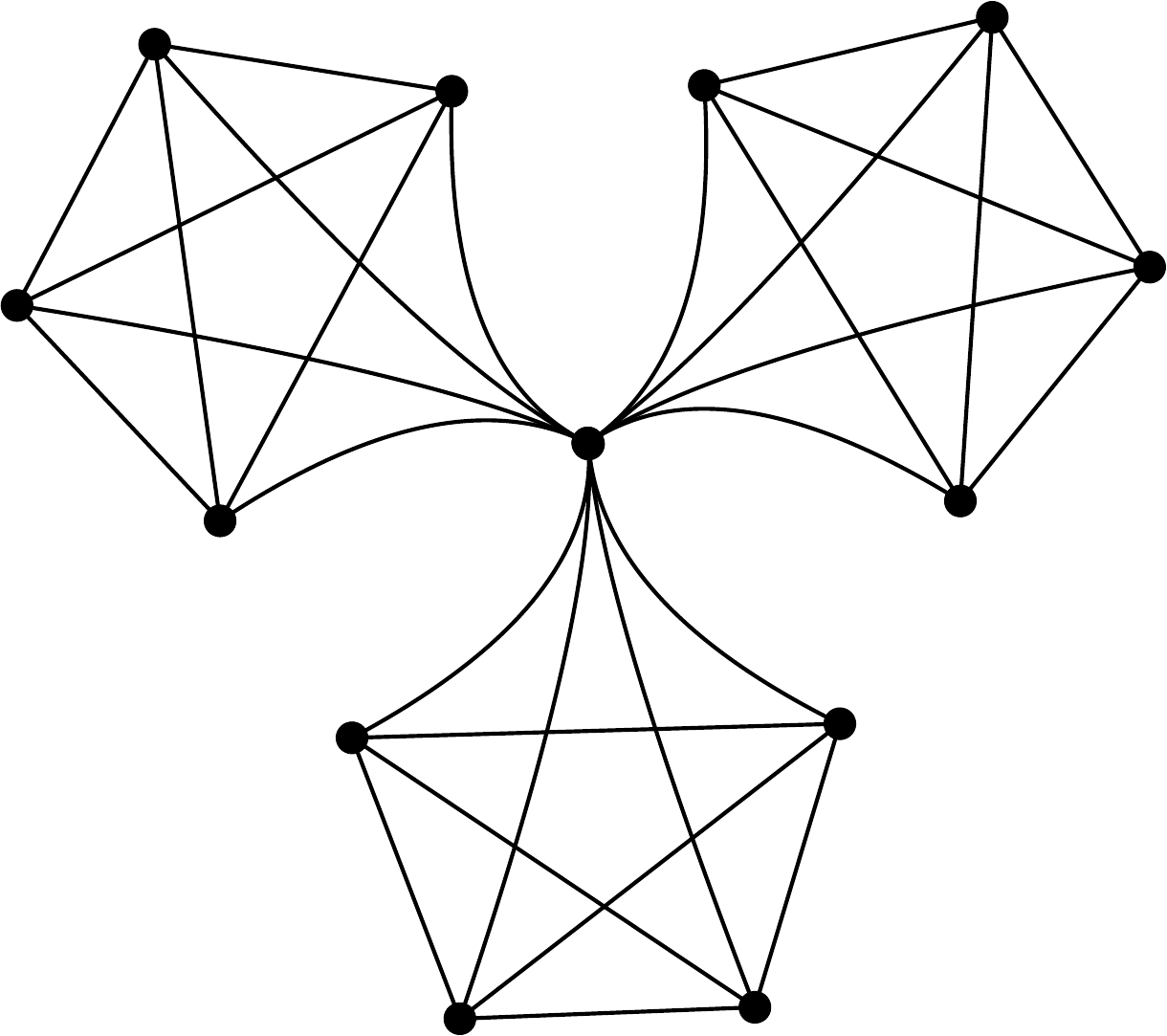} & \includegraphics[page=4,scale=0.12]{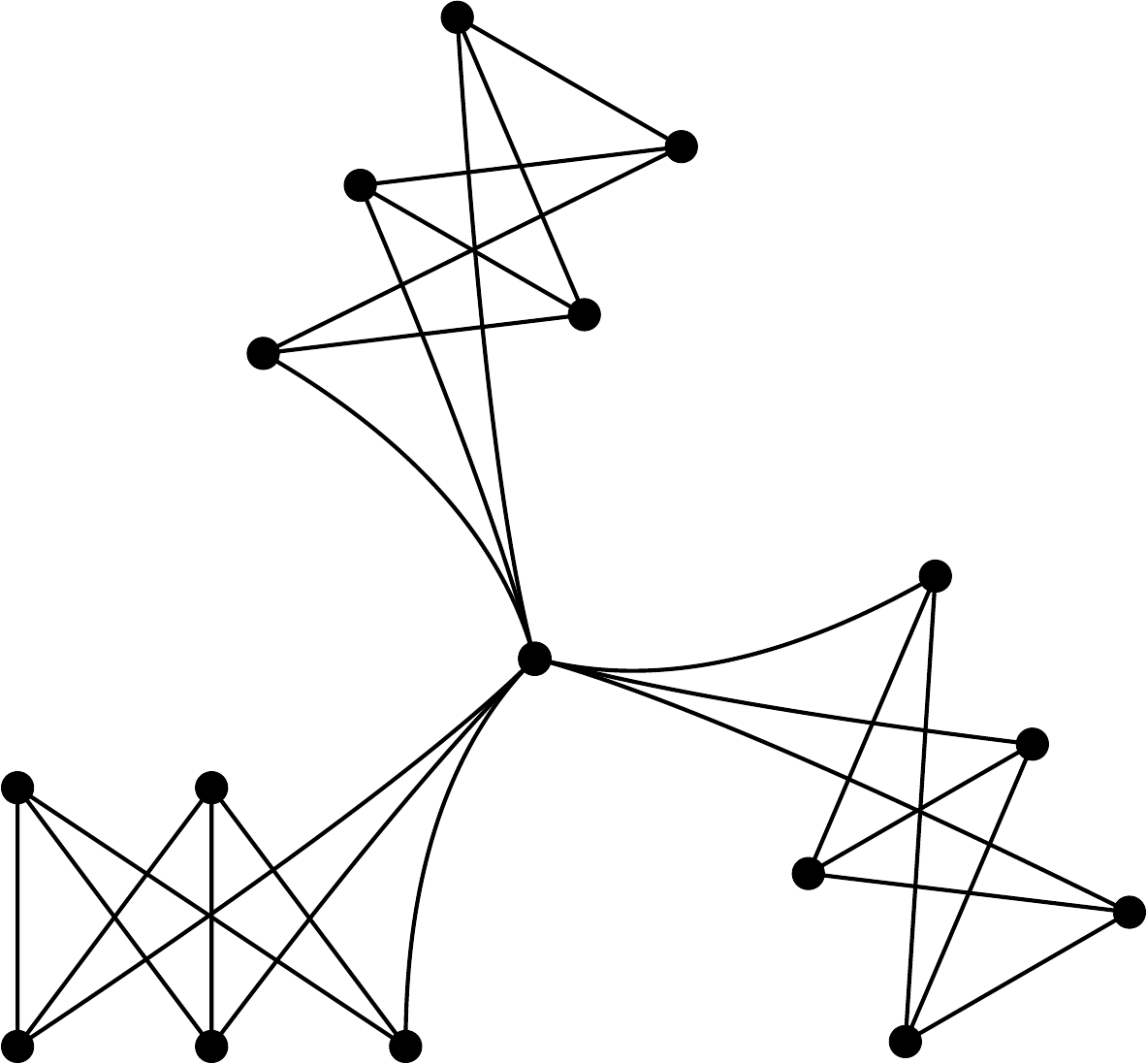} & 
			\includegraphics[page=5,scale=0.11]{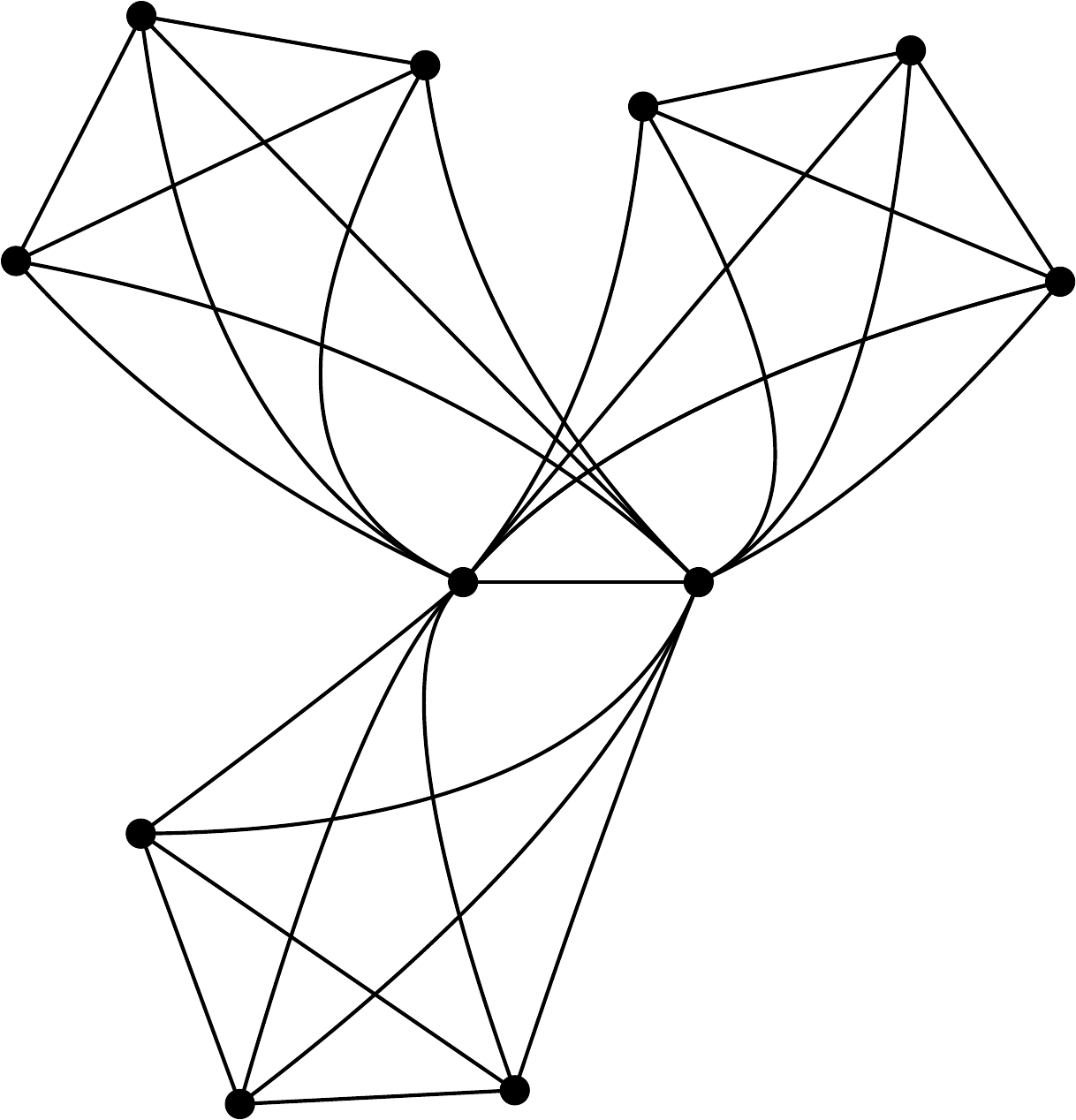} &
			\includegraphics[page=10,scale=0.12]{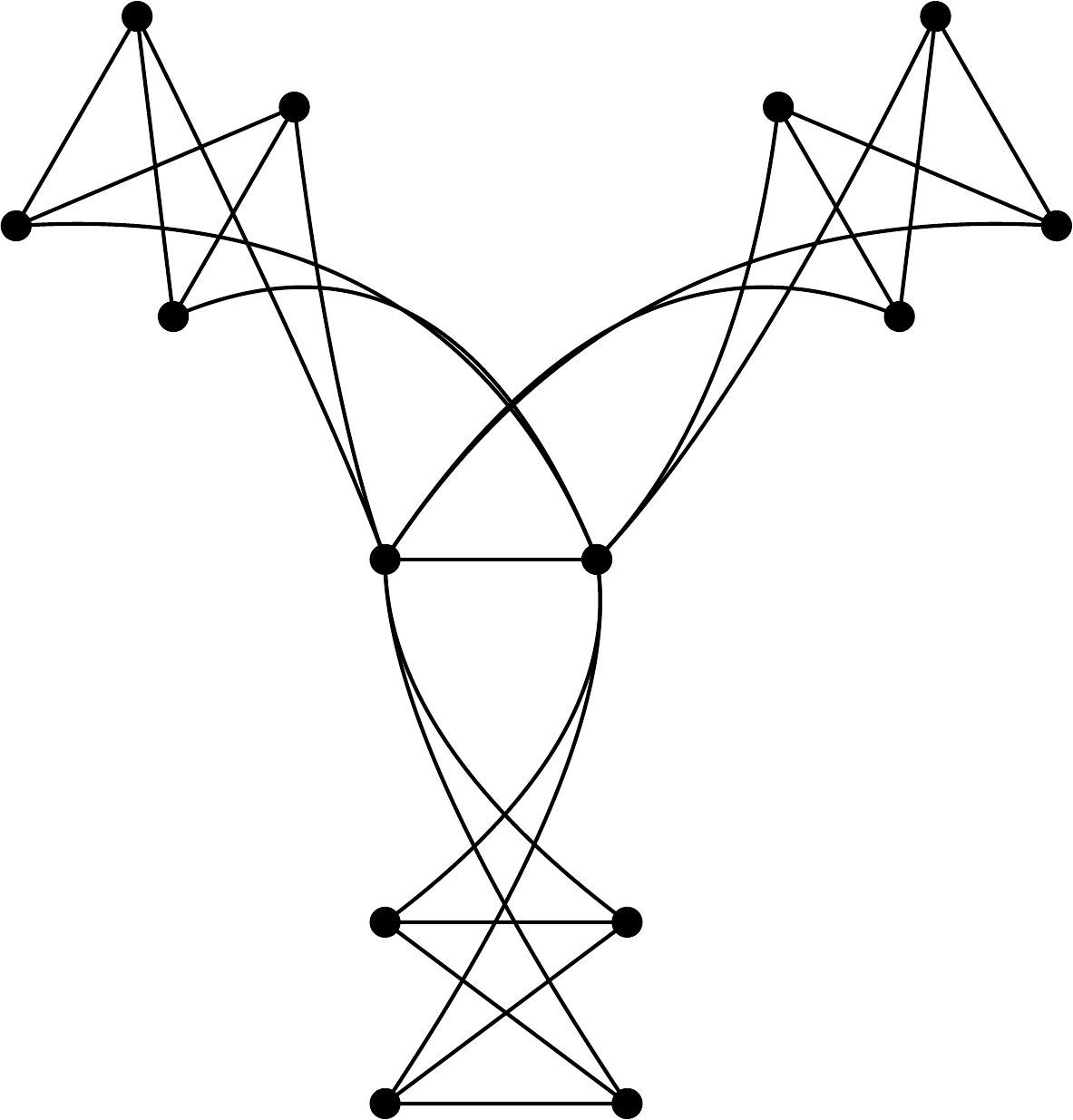} & \includegraphics[page=11,scale=0.12]{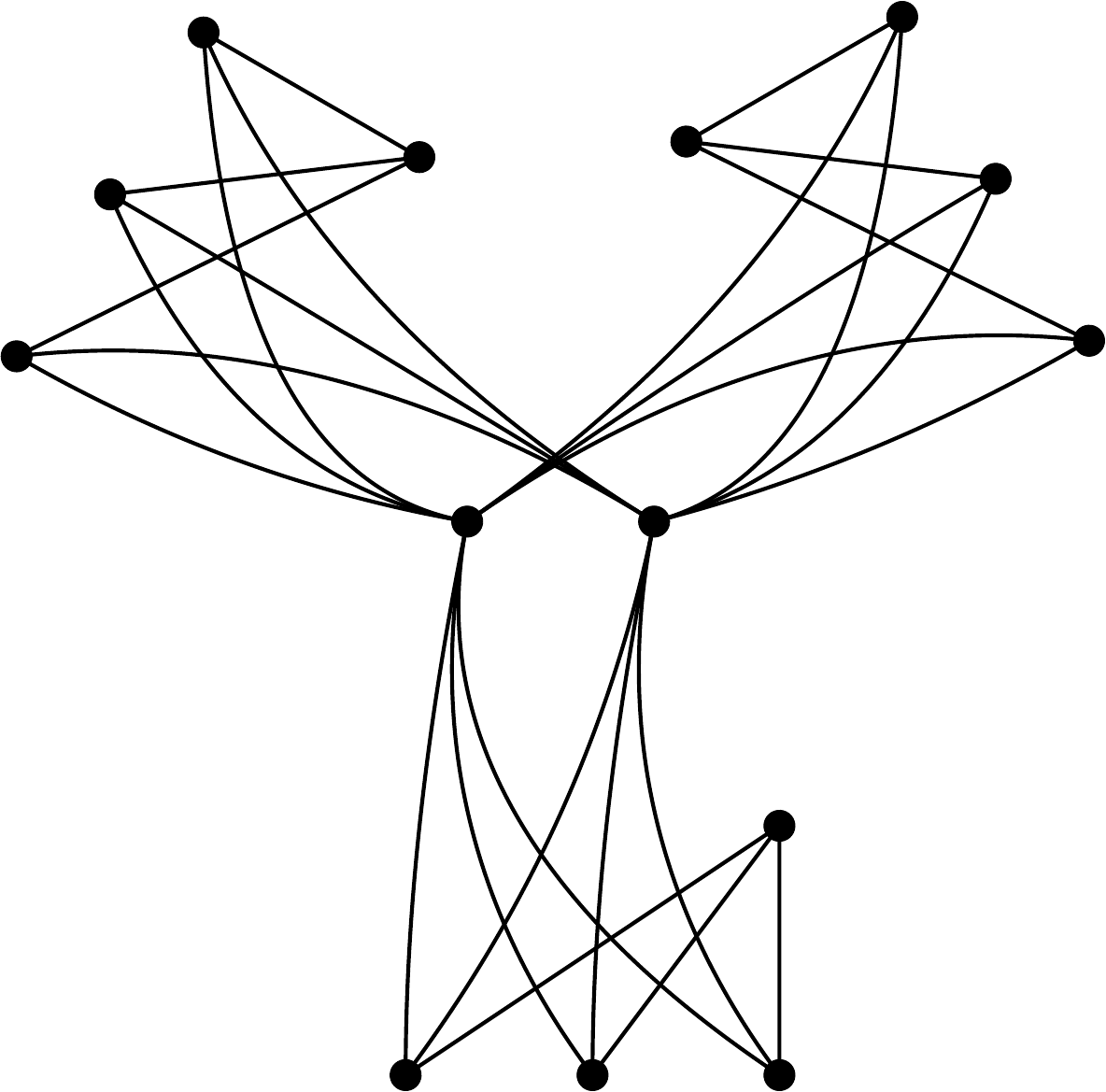} \\
			$\Omega_{1,3}$ &  $\Omega_{2,3}$ & $\Theta_{1,3}$ & $\Theta_{2,3}$ & $\Theta_{3,3}$\\[4mm]
		\end{tabular}
		\setlength{\tabcolsep}{16pt}
		\begin{tabular}{cccc}
			\includegraphics[page=6,scale=0.1]{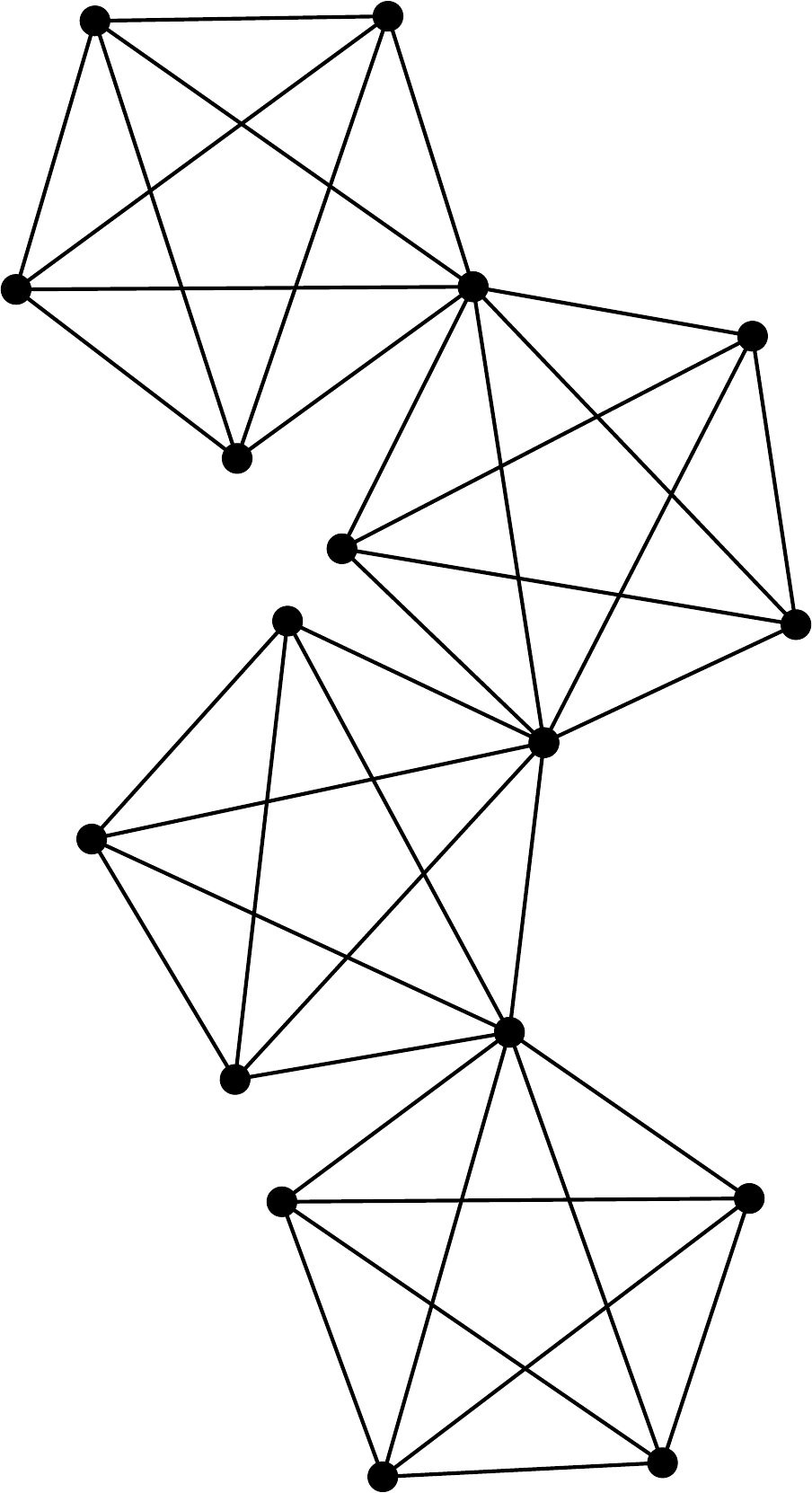}& 
			\includegraphics[page=7,scale=0.11]{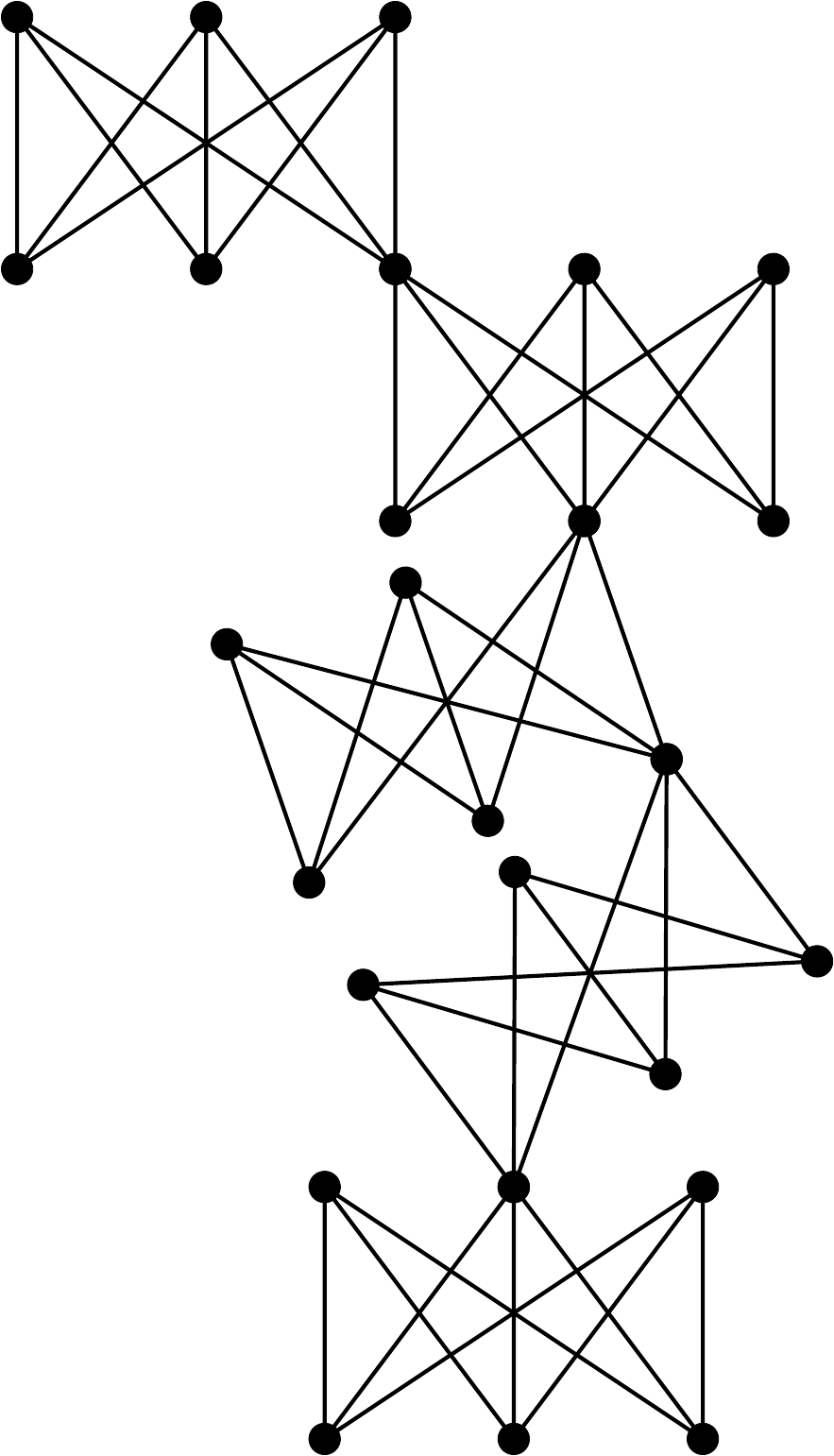} & \includegraphics[page=8,scale=0.13]{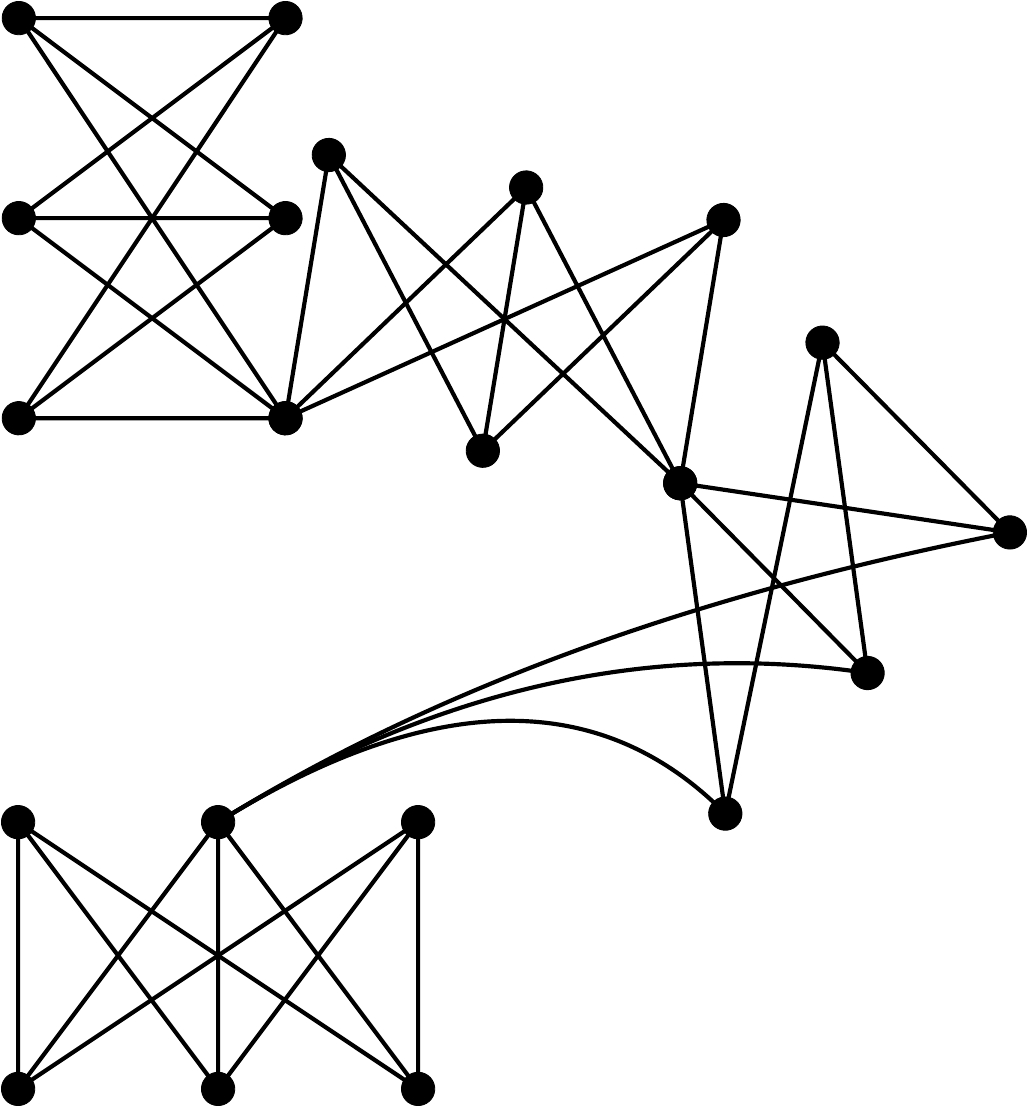} &
			\includegraphics[page=9,scale=0.13]{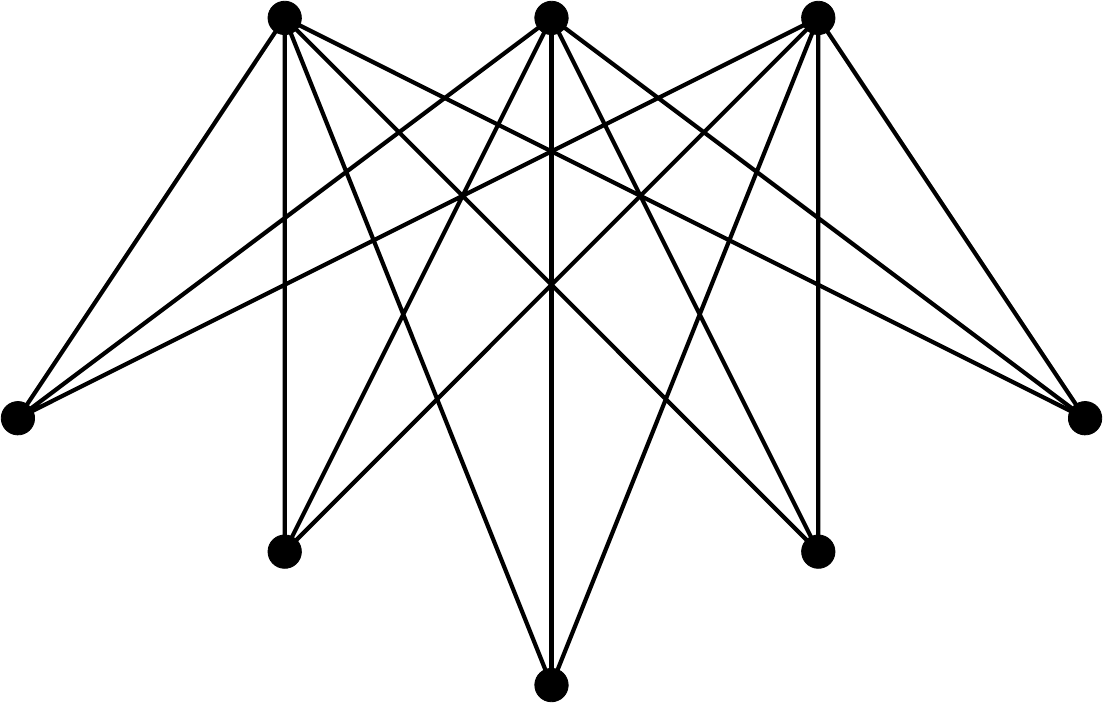}\\
			$\Pi_{1,4}$ & $\Pi_{2,4}$ & $\Pi_{3,4}$ & $K_{3,5}$
		\end{tabular}
		\caption{Some examples of sum-Kuratowski graphs.}
		\label{fig:sum-kuratowski}
	\end{figure}
	
	We will refer to the $\Omega_{i,k}$, $\Theta_{i,k}$ and $\Pi_{i,k}$ graphs collectively the \defn{sum-Kuratowski} graphs, denoted by $\mathcal{K}_k=\{\Omega_{1,k}, \Omega_{2,k}, \Theta_{1,k}, \Theta_{2,k} , \Theta_{3,k} , \Pi_{1,k}, \Pi_{2,k} , \Pi_{3,k} , K_{3,k} \}$. Within a sum-Kuratowski graph, the vertices that are formed by identification are \defn{joining} vertices, and each (connected) component of the graph with joining vertices deleted is a \defn{hanging component}.
	
	For finite $k$, our starting point is the following theorem of Robertson and Seymour \cite{robertson24b}, which can be viewed as a classification of when the graph minor structure theorem \cite{robertson2003graph} can be simplified to being graphs embeddable on a surface. There are analogous theorems that examine when we can avoid vortices \cite{thilikos2024killing}, and when the graphs are embeddable on a surface with a bounded number of additional apex vertices \cite{robertson24a}.
	
	\begin{theorem}[Robertson and Seymour \cite{robertson24b}]\label{thm:RS}
		If a graph $G$ does not contain any of the graphs $\Lambda_{1,k},$ $\Lambda_{2,k},$ $\Omega_{1,k}, \Omega_{2,k}, \Theta_{1,k}, \Theta_{2,k} , \Theta_{3,k} , K_{3,k}$ as a minor, then $G$ is embeddable in a surface of Euler genus at most $f(k)$ where $f:\mathbb{N} \to \mathbb{N}$ is increasing and $\lim_{k \to \infty} f(k)=\infty$.
	\end{theorem}
	
	We first modify this theorem to obtain a statement for connected graphs, in particular removing disjoint unions of Kuratowski subgraphs from consideration at the cost of introducing the path-joined structures $\Pi_{i,k}$. To do this, we need the following simple lemma.
	
	\begin{lemma}\label{lem2:tree}
		Let $T$ be a tree and let $X\subseteq V(T)$ with $|X|\ge k^2$. Either $T$ contains a path $P$ with $|X \cap V(P)| \ge k$, or $T$ contains a subtree $T'$ with leaf set $L$ and $|X\cap L| \ge k$.
	\end{lemma}
	
	\begin{proof}
		By repeatedly deleting leaves of $T$ that are not in $X$, we may assume that all leaves are in $X$. Let $\ell_1, \ldots , \ell_m$ be these leaves. 
		If $m \ge k$, then we may simply take $T'=T$ and $L=\{\ell_1, \ldots , \ell_m\}$. 
		Otherwise, suppose that $m <k$. 
		For each $2\le i \le m$, let $P_i$ be the path contained in $T$ between $\ell_1$ and $\ell_i$.
		Then $\bigcup_{i=2}^m V(P_i) = V(T)$. Therefore, there exists some $2\le i \le m$ with $|V(P_i) \cap X| \ge |X|/(m-1) \ge |X|/k\ge k$, as desired.
	\end{proof}
	
	\begin{lemma}\label{lem2:connect}
		Let $G$ be a connected graph containing either $\Lambda_{1,9k^2}$ or $\Lambda_{2,9k^2}$ as a minor.
		Then $G$ contains one of $\Omega_{1,k}, \Omega_{2,k}, \Pi_{1,k}, \Pi_{2,k}, \Pi_{3,k}$ as a minor.
	\end{lemma}
	
	\begin{proof}
		By contracting edges of $G$, we may assume that $G$ contains either $\Lambda_{1,9k^2}$ or $\Lambda_{2,9k^2}$ as a subgraph.
		Let $R_1, \ldots , R_{9k^2}$ be vertex disjoint subgraphs of $G$, all of which are either $K_5$ or $K_{3,3}$.
		Now, contract the edges of each $R_i$ to vertices $r_i$ and call the resulting graph $G'$.
		Let $T$ be a spanning tree of $G'$ and let $X=\{r_1,\ldots ,r_{9k^2}\}$.
		
		Suppose that there exists a subtree $T'$ of $T$ with leaves $L$ and $|X\cap L|=k$.
		Then $G$ contains a subgraph $H$ which can be obtained from a tree with $k$ leaves by identifying each leaf with a vertex of a different $K_5$ or $K_{3,3}$ (depending on whether all $R_i$ are $K_5$ or $K_{3,3}$).
		In the first case we obtain $\Omega_{1,k}$ as minor by contracting this tree, and otherwise we obtain $\Omega_{2,k}$.
		So we may assume that there is no subtree $T'$ of $T$ with leaves $L$ and $|X\cap L|=k$.
		
		By \cref{lem2:tree}, $T$ contains a path $P$ with $|V(P)\cap X|=3k^2$.
		Suppose first that each $R_i$ is a $K_5$.
		Then in $G$ there is a path $P'$ that either intersects $k$ copies of $K_5$, each in a single vertex, or each in two vertices being consecutive vertices of the path.
		In the first case, $G$ contains $\Omega_{1,k}$ as a minor by contracting the path $P'$.
		In the second case, $G$ contains $\Pi_{1,k}$ as a minor by contracting the edges of the path $P'$ not contained in one of the $k$ copies of $K_5$. 
		Thus, we may assume that each $R_i$ is a $K_{3,3}$.
		Then in $G$ there is a path $P'$ that either intersects $k$ copies of $K_{3,3}$, each in a single vertex, or each in two vertices being consecutive vertices of the path and adjacent in the $K_{3,3}$, or lastly each in three vertices being consecutive vertices of the path and being a 3-vertex path in of $K_{3,3}$.
		In the first case, $G$ contains $\Omega_{2,k}$ as a minor by contracting the path $P'$.
		For the remaining cases, we can contract the edges of the path $P'$ not contained in one of the $k$ copies of $K_{3,3}$ to see that $G$ contains $\Pi_{2,k}$ or $\Pi_{3,k}$ respectively as a minor.
		
		We conclude that $G$ contains one of $\Omega_{1,k}, \Omega_{2,k}, \Pi_{1,k}, \Pi_{2,k}, \Pi_{3,k}$ as a minor, as desired.
	\end{proof}
	
	Using \cref{lem2:connect}, the following interim theorem follows directly from \cref{thm:RS}.
	
	\begin{theorem}\label{thm:connectedRS}
		Every connected graph $G$ containing no $H$-minor for any $H\in \mathcal{K}_k$ is embeddable in a surface of Euler genus at most $f(k)$.
	\end{theorem}
	
	There is a parallel sequence of results in the countable case, starting with the following characterisation of graphs embeddable in a surface (with unrestricted genus).
	\begin{theorem}[Georgakopoulos \cite{georgakopoulos2023excluded}]\label{thm:Georgakopoulos}
		A countable graph $G$ embeds into a compact (orientable) surface if and only if it does not have any of $\Lambda_{1,\infty}$, $\Lambda_{2,\infty}$, $\Omega_{1,\infty}$, $\Omega_{2,\infty}$, $\Theta_{1,\infty}$, $\Theta_{2,\infty}$, $\Theta_{3,\infty}$, or $K_{3,\infty}$ as a minor.
	\end{theorem}
	
	We again derive a modified version from this in the same manner as for the finite case.
	\begin{lemma}\label{lem2:treeinf}
		Let $T$ be a tree and let $X\subseteq V(T)$ be infinite.
		Then, either $T$ contains a one-way infinite path $P$ with $|X \cap V(P)|=\infty$, or $T$ contains a subtree $T'$ with leaves $L$ and $|X\cap L| = \infty$.
	\end{lemma}
	
	\begin{proof}
		Fix any root vertex $r$ of $T$, and delete each vertex $v$ of $T$ that does not lie on a path between $r$ and an element of $X$. 
		The resulting graph is a subtree $T'$ containing $r$ in which every one-way infinite path starting at $r$ contains infinitely many vertices of $X$, and every leaf of $T'$ (except possibly $r$) is contained in $X$. 
		Thus, we are done if $T$ has a one-way infinite path or if $|L| =\infty$. If neither of these is true, then $T$ must be a finite tree, but this contradicts the fact that $|X|=\infty$. 
	\end{proof}
	
	\begin{lemma}\label{lem2:connectinf}
		Let $G$ be a connected countable graph containing either $\Lambda_{1,\infty}$ or $\Lambda_{1,\infty}$ as a minor.
		Then $G$ contains one of $\Omega_{1,\infty}$, $\Omega_{2,\infty}$, $\Pi_{1,\infty}$, $\Pi_{2,\infty}$, or $\Pi_{3,\infty}$ as a minor.
	\end{lemma}
	
	\begin{proof}
		By contracting edges of $G$, we may assume that $G$ contains either $\Lambda_{1,\infty}$ or $\Lambda_{2,\infty}$ as a subgraph.
		Let $(R_i:)_{i\in \mathbb{N}}$ be vertex disjoint subgraphs of $G$, all of which are either $K_5$ or $K_{3,3}$.
		Now, contract the edges of each $R_i$ to vertices $r_i$, to obtain a new graph $G'$.
		Let $T$ be a spanning tree of $G'$ and let $X=\{r_i : i\in \mathbb{N}\}$.
		
		Suppose that there exists a subtree $T'$ of $T$ with leaves $L$ and $|X\cap L|=\infty$.
		Then $G$ contains as a subgraph a graph $H$ obtained from a tree with an infinte number of leaves by identifying each leaf with a vertex of different $K_5$ or $K_{3,3}$ depending on whether all $R_i$ are $K_5$ or $K_{3,3}$.
		In the first case we obtain $\Omega_{1,\infty}$ as minor by contracting this tree, and otherwise we obtain $\Omega_{2,\infty}$.
		So we may assume that there is no subtree $T'$ of $T$ with leaves $L$ and $|X\cap L|=\infty$.
		
		By \cref{lem2:treeinf}, $T$ contains a one-way infinite path $P$ with $|V(P)\cap X|=\infty$.
		Suppose first that each $R_i$ is a $K_5$.
		Then, in $G$ there is a one-way infinite path $P'$ that intersects an infinite number of copies of $K_5$, either each in a single vertex, or each in two vertices (so consecutive vertices of the path).
		In the first case, $G$ contains $\Omega_{1,\infty}$ as a minor by contracting the path $P'$.
		In the second case, $G$ contains $\Pi_{1,\infty}$ as a minor by contracting the edges of the path $P'$ not contained in one of the infinite number of copies of $K_5$.
		Thus, we may assume that each $R_i$ is a $K_{3,3}$.
		Then there is a path $P'$ in $G$ that intersects an infinite number of copies of $K_{3,3}$, either each in a single vertex, or each in two vertices being consecutive vertices of the path and adjacent in the $K_{3,3}$, or lastly each in three vertices being consecutive vertices of the path and being a 3-vertex path in of $K_{3,3}$.
		In the first case, $G$ contains $\Omega_{2,\infty}$ as a minor by contracting the path $P'$.
		For the remaining two cases, by contracting the edges of the path $P'$ not contained in one of the infinite number of copies of $K_{3,3}$, we see that $G$ has a $\Pi_{2,\infty}$ or $\Pi_{3,\infty}$ minor respectively.
		Thus, $G$ contains one of $\Omega_{1,\infty}. \Omega_{2,\infty}, \Pi_{1,\infty}, \Pi_{2,\infty}, \Pi_{3,\infty}$ as a minor, as desired.
	\end{proof}
	
	By \cref{thm:Georgakopoulos} and \cref{lem2:connectinf}, we obtain a connected version of \cref{thm:Georgakopoulos}. 
	
	\begin{theorem}\label{thm:connGeorgakopoulos}
		A countable connected graph $G$ embeds into a compact (orientable) surface if and only if it does not have any of $\Omega_{1,\infty}$, $\Omega_{2,\infty}$, $\Theta_{1,\infty}$, $\Theta_{2,\infty}$, $\Theta_{3,\infty}$, $\Pi_{1,\infty}$, $\Pi_{2,\infty}$, $\Pi_{3,\infty}$, or $K_{3,\infty}$ as a minor.
	\end{theorem}
	
	The graph minor relation is strictly contained in the Y-minor relation, so the theorems above are still true if we simply replace `minor' with `Y-minor'. However, when we allow the additional operations of the latter, we can also significantly reduce the family of excluded Y-minors to just the sum-Kuratowski graphs built from $K_5$'s.
	\begin{lemma}\label{lem:Yminorreductions}
		For every $k\in \mathbb{N}\cup \infty$, we have the following:
		\begin{itemize}
			\item $\Omega_{2,k}$ contains $\Omega_{1,k}$ as a Y-minor,
			\item $\Theta_{2,k}$ contains $\Theta_{1,k}$ as a Y-minor,
			\item $\Theta_{3,k}$ contains $\Theta_{1,k}$ as a Y-minor,
			\item $\Pi_{2,k}$ contains $\Pi_{1,k}$ as a Y-minor, and
			\item $\Pi_{3,k}$ contains $\Pi_{1,k}$ as a Y-minor.
		\end{itemize}
	\end{lemma}
	
	\begin{proof}
		Let $H\in \{\Omega_{2,k}, \Theta_{2,k}, \Theta_{3,k}, \Pi_{2,k}, \Pi_{3,k}\}$.
		It is enough to show that each $K_{3,3}$ of $H$ can be replaced with a $K_5$ that contains the same joining vertices.
		Let $a_1,a_2,a_3,b_1,b_2,b_3$ be the vertices of one such $K_{3,3}$ with bipartition $\{a_1,a_2,a_3\}$, $\{b_1,b_2,b_3\}$ and where neither $a_3$ nor $b_3$ is a joining vertex of $H$.
		Then both $a_3$ and $b_3$ have degree three in $H$.
		Therefore, we can add the edges $a_1a_2,b_1b_2,b_2b_3,b_1b_3$ and remove the vertex $a_3$ to obtained as a Y-minor the graph obtained from $H$ by replacing this $K_{3,3}$ with a $K_5$ containing the same joining vertices.
		Doing this to all copies of $K_{3,3}$ in $H$ gives the desired Y-minor of $H$.
	\end{proof}
	
	Applying \cref{lem:Yminorreductions} to both \cref{thm:connectedRS} and \cref{thm:connGeorgakopoulos}, the final Kuratowski-like theorems that we take forward are the following.
	
	\begin{theorem}\label{thm:Yminorfinite}
		Every connected graph $G$ that does not contain any of $\Omega_{1,k}$, $\Theta_{1,k}$,
		$\Pi_{1,k}$, 
		or $K_{3,k}$ as a Y-minor has Euler genus at most $f(k)$.
	\end{theorem}
	
	\begin{theorem}\label{thm:Yminorcountable}
		A countable connected graph $G$ embeds into a compact (orientable) surface if and only if it does not have any of $\Omega_{1,\infty}$, $\Theta_{1,\infty}$,
		$\Pi_{1,\infty}$, or $K_{3,\infty}$ as a Y-minor.
	\end{theorem}
	
	At this point, we are just missing one more key result connecting the above theorems to bounds on ply for our main results. Let us state this below, leaving the proof to the next section, and then demonstrate how \cref{thm:finply2} and \cref{thm:finply_countable} follow.
	
	\begin{lemma}\label{lem:3main}
		Let $\Sigma$ be a surface of Euler genus $g$ and $H$ be one of $\Omega_{1,k}$, $\Theta_{1,k}$, $\Pi_{1,k}$, or $K_{3,k}$ for some $k \geq 25$.
		Then, every finite $\Sigma$-cover $\widehat{H}$ of $H$ has ply at most $\frac{28g}{k-24}$.
	\end{lemma}
	
	Suppose we fix a surface $\Sigma$ of Euler genus $g$. To prove \cref{thm:finply2}, choose $c = f(25)$ (where $f(k)$ is the function in \cref{thm:RS}) and let $G$ be a connected graph with Euler genus $h \geq c$ and a finite $p$-ply $\Sigma$-cover. By \cref{thm:Yminorfinite}, $G$ must contain a Y-minor $H$ that is isomorphic to one of $\Omega_{1,k}$, $\Theta_{1,k}$, $\Pi_{1,k}$, or $K_{3,k}$ where $k = f^{-1}(h) \geq 25$. In all cases, we know from \cref{prop:yminorcover} that $H$ also has a $p$-ply $\Sigma$-cover, so by \cref{lem:3main} we must have $p \le \frac{28g}{k-24} = \frac{28g}{f^{-1}(h)-24}$. \cref{thm:finply2} now follows assuming \cref{lem:3main} since $f^{-1}$ is increasing and $\lim_{h\to \infty} f^{-1}(h) = \infty$.
	
	Following the same outline for \cref{thm:finply_countable}, let $G$ be a connected countable graph that does not embed in a compact surface. By \cref{thm:Yminorcountable}, this is true if and only if $G$ contains some $H \in \{\Omega_{1,\infty}, \Theta_{1,\infty}, \Pi_{1,\infty}, K_{3,\infty}\}$ as a Y-minor, but then taking $k\to \infty$ in \cref{lem:3main} implies that $H$ (and hence $G$, by \cref{prop:yminorcover}) has no finite ply $\Sigma$-cover for any given compact surface $\Sigma$. Hence $G$ has no finite ply cover that embeds into a compact surface.
	
	\section{Bounding ply for excluded Y-minors}
	\label{sec:sumkuratowski}
	
	The goal of this section is to prove \cref{lem:3main}. This boils down to showing that if $H$ is one of the excluded Y-minors in \cref{thm:Yminorfinite} or \cref{thm:Yminorcountable}, then the ply of any $\Sigma$-cover $\widehat{H}$ of $H$ is suitably bounded in terms of the Euler genus of $\Sigma$. In all cases, the basic idea is that, since covers are bijective on neighbourhoods, one can easily relate the ply to the number of edges in $\widehat{H}$ and number of vertices in $\widehat{H}$. Euler's formula provides a means to bound these quantities, but to apply it effectively, we will need to consider embedded multigraphs and their minors.

	A \defn{face} of a multigraph $G$ embedded in a surface $\Sigma$ is a component of $\Sigma - G$ homeomorphic to an open disk. A \defn{2-face} is a face bounded by just 2 edges.
	We say that two edges in an embedded multigraph are \defn{parallel} if they form a 2-face in the embedding.
	In this section, when taking minors of embedded graphs, we only delete multiedges in the (multi)graph after an edge contraction so that there is no pair of parallel edges (in other words, so that the embedded (multigraph) has no 2-face).
	This is to ensure that good bounds can be obtained from Euler's formula, as below (see Section 3.1 of \cite{mohar2001graphs}).

	\begin{theorem}\label{Euler}
		Let $G$ be a $n$-vertex multigraph embedded on a surface of Euler genus $g$ without 2-faces.
		Then $|E(G)|\le 3n-6+3g$.
	\end{theorem}
	
	\begin{theorem}\label{EulerBi}
		Let $G$ be a $n$-vertex bipartite multigraph embedded on a surface of Euler genus $g$ without 2-faces.
		Then $|E(G)|\le 2n-4+2g$.
	\end{theorem}

	This is already enough to handle the easiest graph in our collection; $K_{3,k}$.
	
	\begin{lemma}\label{lem:K3}
		Let $\Sigma$ be a surface of Euler genus $g$ and let $k\ge 7$.
		If $K_{3,k}$ has a finite $p$-ply $\Sigma$-cover, then $p \le \frac{2g}{k-6}$.
	\end{lemma}

	\begin{proof}
		Let $\widehat{H}$ be a $p$-ply cover of $K_{3,k}$ embedded in a surface of Euler genus $g$.
		Then $|V(\widehat{H})|=p(k+3)$ and $|E(\widehat{H})|=3pk$.
		Since $K_{3,k}$ is bipartite, so is $\widehat{H}$.
		Also, since $\widehat{H}$ is a simple graph, the embedding of $\widehat{H}$ has no 2-faces.
		Therefore, by \cref{EulerBi}, we have $|E(\widehat{H})| \le 2|V(\widehat{H})| - 4 + 2g$, and so $3pk\le 2p(k+3) - 4 +2g\le 2p(k+3) +2g$.
		This gives $p\le \frac{2g}{k-6}$, as desired.
	\end{proof}
	
	Surfaces are locally planar by definition, so if we `cut up' the cover $\widehat{H}$, we can obtain planar covers of minors of $H$. A rather useful fact is that none of these can contain $\Omega_{1,2}$ nor $\Theta_{1,2}$ as a minor by the following result of Negami \cite{negami1988} and Archdeacon (see \cite{hlineny201020} for a sketch of the proof).
	
	\begin{theorem}
		\label{noplanarcover}
		Neither $\Omega_{1,2}$ nor $\Theta_{1,2}$ have finite planar covers.
	\end{theorem}

	A component of an embedded graph is \defn{non-contractible} if it contains a non-contractible cycle.
	If a component is not non-contractible, then it is \defn{contractible}.
	Recall that the hanging components of a sum-Kuratowski graph $H$ are the components left when all joining vertices are deleted. 
	For $\Omega_{1,k}$, $\Theta_{1,k}$, and $\Pi_{1,k}$, we will first work toward a bound on ply when most vertices are in non-contractible components of the lift of hanging components.
	The next lemma relies on topological facts that appear in a paper of Malni\v{c} and Mohar \cite{malnivc1992generating}. We defer to their paper for topological definitions.
	
	\begin{lemma}\label{lem:boundingnoncontractible}
		Let $\Sigma$ be a surface of Euler genus $g$.
		For any $k\ge 4$, let $H\in \{\Omega_{1,k}, \Theta_{1,k}, \Pi_{1,k} \}$ and let $\mathcal{C}$ be the hanging components of $H$. If $\widehat{H}$ is a finite $\Sigma$-cover of $H$, then $\widehat{\mathcal{C}}$ has at most $6g$ components that are non-contractible in $\Sigma$.
	\end{lemma}
	
	\begin{proof}
		Take an embedding of $\widehat{H}$ in $\Sigma$, and suppose for a contradiction that $\widehat{\mathcal{C}}$ has at least $6g+1$ non-contractible components.
		For each non-contractible component $C$ of $\widehat{\mathcal{C}}$, let $J_C$ be a non-contractible cycle contained in $\widehat{\mathcal{C}}$.
		Then let $\mathcal{J}$ be the collection of these non-contractible cycles, noting that the elements in $\mathcal{J}$ are necessarily disjoint.
		A surface of Euler genus $g$ cannot contain $3g+1$ disjoint pairwise non-homotopic non-contractible cycles \cite[Proposition 3.7]{malnivc1992generating}.
		Therefore, $\mathcal{J}$ contains three pairwise disjoint homotopic non-contractible cycles $J_1,J_2,J_3$.
		Since these cycles are disjoint and homotopic, they are two-sided \cite[Proposition 3.2]{malnivc1992generating}.
		After possibly relabelling $J_1$, $J_2$, and $J_3$, it follows that $\Sigma$ contains a cylinder $\mathbb{S}^1 \times [0,1]$ with $J_2$ in its interior and whose boundary components $\mathbb{S}^1 \times \{0\}$ and $\mathbb{S}^1 \times \{1\}$ consist of $J_1$ and $J_3$ respectively \cite[Proposition 3.5]{malnivc1992generating}.
		
		For each $1\le i \le 3$, let $C_i$ be the hanging component of $H$ such that $J_i$ is contained in the lift of $C_i$.
		Since $k\geq 4$, the subgraph $H \backslash \{C_1 \cup C_3\}$ contains $\Omega_{1,2}$ or $ \Theta_{1,2}$ as a minor, and therefore has no finite planar cover by \cref{noplanarcover}.
		Let $x$ be a joining vertex of $H$ adjacent to a vertex of $C_2$.
		Then, there is a vertex $\widehat{x}$ in the interior of the cylinder adjacent to a vertex of $\widehat{C_2}$ (containing $J_2$).
		Since $H\backslash \{C_1 \cup C_3\}$ is connected and $x \in V(H\backslash \{C_1 \cup C_3\})$, it follows that one component of the lift of $H\backslash \{C_1 \cup C_3\}$ is contained in the cylinder.
		But this component is a cover of $H\backslash \{C_1 \cup C_3\}$, which is a contradiction.
	\end{proof}

	\begin{lemma}
		\label{lem:finply_noncontractible}
		Let $\Sigma$ be a surface of Euler genus $g$, let $k\ge 5$, let $H\in \{\Omega_{1,k},  \Theta_{1,k},  \Pi_{1,k} \}$, and let $\mathcal{C}$ be its hanging components. If $\widehat{H}$ is a finite $p$-ply $\Sigma$-cover of $H$ and at least half of the vertices of $\widehat{\mathcal{C}}$ are contained in non-contractible components, then $p\le \frac{28g}{k -4}$.
	\end{lemma}
	
	\begin{proof}
		Let $C_1,\ldots, C_r$ be the non-contractible components of $\widehat{\mathcal{C}}$.
		By \cref{lem:boundingnoncontractible}, we have that $r\le 6g$.
		Let $P$ be the path of joining vertices of $H$, and let $P_1,\ldots, P_p \subset \widehat{H}$ be the components of the lift of $P$.
		Consider the minor $H^*$ of $\widehat{H}$ obtained by performing the following operations in order:
		\begin{enumerate}
			\item Delete all of the contractible components of $\widehat{\mathcal{C}}$;
			\item For each $1\le i \le r$, contract the component $C_i$ to a vertex $c_i$;
			\item For each $1\le i \le p$, contract the path $P_i$ to a vertex $x_i$.
		\end{enumerate}
		The resulting minor $H^*$ embedded in $\Sigma$ is bipartite with bipartition $\{c_i : 1\le i \le r\} \cup \{x_i : 1\le i \le p\}$, and we can see that $|V(H^*)|=r+p\le p+6g$. By \cref{EulerBi}, we then have $|E(H^*)|\le 2(p+6g) -4 +2g= 2p -4 +14g$.
		Since at least half of the vertices of $\widehat{\mathcal{C}}$ are contained in non-contractible components, the average degree of the vertices $\{x_i : 1\le i \le p\}$ is at least $k/2$.
		Thus, $|E(H^*)|\ge pk/2$, so $pk/2 \le 2p -4 +14g \le 2p +14g$.
		With $k\geq 5$, this gives $p\le \frac{28g}{k -4}$ as claimed.
	\end{proof}
	
	The preceding statement now allows us to handle $\Omega_{1,k}$ straightforwardly.
	
	\begin{lemma}
		\label{lem:omega}
		Let $\Sigma$ be a surface of Euler genus $g$ and let $k\ge 5$.
		If $\Omega_{1,k}$ has a finite $p$-ply $\Sigma$-cover, then $p\le \frac{28g}{k -4}$.
	\end{lemma}
	
	\begin{proof}
		Let $\widehat{H}$ be a $p$-ply cover of $\Omega_{1,k}$ embedded in $\Sigma$, and let $\vphi$ be the covering map.
		By \cref{lem:finply_noncontractible}, it is enough to show that no hanging component of $H$ lifts to a contractible cycle.
		Suppose otherwise, so there is a subgraph $G \subset \widehat{H}$ that is contractible in $\Sigma$ and $\vphi(G) = C$ for some hanging component $C$ of $H$. It follows that $G$ is a cover of $K_4$.
		Let $X \subset V(G)$ be the vertices incident to the outer face containing the rest of the surface.
		Since $G$ is a cover of $K_4$, it is $3$-regular and hence not outerplanar, so there exists a vertex $v\in V(G)\backslash X$ with $\vphi(v)=c$ for some $c\in C$.
		Therefore, there is a vertex $x$ in the interior of the embedding of $G$ such that $\vphi(x)$ is the joining vertex in $\Omega_{1,k}$.
		This means that some component of the lift of $\Omega_{1,k}\backslash C= \Omega_{1,k-1}$ is contained in the interior of the embedding of $G$.
		But this contradicts \cref{noplanarcover}, which states that $\Omega_{1,2}$ has no finite planar cover.
	\end{proof}

	The remaining two sum-Kuratowski graphs can have hanging components whose lifts are contractible, so we need an extra step of analysing the contractible components. In particular, we show that the contractible components of the lifts of hanging components (which are now triangles) can be assumed to be facial triangles. 
	
	\begin{lemma}\label{facial}
		Let $\Sigma$ be a surface of Euler genus $g$, let $k\ge 3$, and
		let $H\in \{\Theta_{1,k}, \Pi_{1,k} \}$.
		If $H$ has $p$-ply $\Sigma$-cover, then there is a $p$-ply $\Sigma$-cover $\widehat{H}$ of $H$ such that for every hanging triangle $C$ of $H$, every contractible component of $\widehat{C}$ is a facial triangle.
	\end{lemma}
	
	\begin{proof}
		Choose $\widehat{H}$ to be a $p$-ply $\Sigma$-cover of $H$ minimising the number of contractible components of $\widehat{C}$ that are not facial triangles, and let $\vphi$ be the covering map.
		Suppose for contradiction that some contractible component $B$ of $\widehat{C}$ is not a facial triangle.
		This $B$ is a cover of a hanging triangle $C \subset H$, say with $V(C) = \{x_1, x_2 , x_3\}$, so $B$ is a cycle with vertices $y_1,\ldots , y_{3r}$ in order where for every pair $1\le i \le 3$ and $1\le j \le r$ we have $\vphi(y_{3(r-1) + i}) = x_i$.
		
		We claim that $B$ is a facial cycle.
		If not, then $B$ is a separating cycle and therefore there is a vertex $v$ of $\widehat{H}$ in the interior of $B$ that projects onto a joining vertex of $H$.
		Then in the interior of $B$ there is planar component of $\widehat{H}$ that covers $H\backslash C$.
		Since $H\backslash C$ contains either $\Theta_{1,2}$ or $\Omega_{1,2}$ as a minor, this contradicts \cref{noplanarcover} and so proves the claim.
		
		Now, consider $\widehat{H^*}=(\widehat{H} \backslash \{y_{3j}y_{3j+1} : 1\le j \le r\} )\cup \{y_{3(j-1)+1}y_{3j} : 1\le j \le r\}$. Since $B$ is facial, we observe that $\widehat{H^*}$ is a $p$-ply $\Sigma$-cover of $H$ with fewer contractible components of $\widehat{C}$ that are not a facial triangles.
		This contradicts our choice of $\widehat{H}$.
	\end{proof}
	
	We are now ready to handle the last two sum-Kuratowski graphs; $\Theta_{1,k}$ and then $\Pi_{1,k}$.
	
	\begin{lemma}
		\label{lem:theta}
		Let $\Sigma$ be a surface of Euler genus $g$ and let $k\ge 9$.
		If $\Theta_{1,k}$ has a finite $p$-ply $\Sigma$-cover, then $p\le \frac{28g}{k -8}$.
	\end{lemma}

	\begin{proof}
		Let $\widehat{H}$ be a $p$-ply cover of $\Theta_{1,k}$ embedded in $\Sigma$.
		Let $\mathcal{C}$ be the hanging triangles of $\Theta_{1,k}$.
		By \cref{facial}, we may assume that for every hanging triangle $C$ of $H$, every contractible component of $\widehat{C}$ is a facial triangle.
		By \cref{lem:finply_noncontractible}, we may assume that at least half of the vertices of $\widehat{\mathcal{C}}$ are contained in contractible components.
		
		Let $x$ and $y$ be the joining vertices of $\Theta_{1,k}$.
		Let $x_1, \ldots , x_p$ be the vertices of $\widehat{H}$ that project onto $x$ and let $y_1, \ldots , y_p$ be the vertices of $\widehat{H}$ that project onto $y$.
		Let $C_1,\ldots , C_r$ be the contractible (facial triangle) components of $\widehat{\mathcal{C}}$.
		Then $kp/2 \le r$.
		Now, let $G$ be the embedded graph obtained from $\widehat{H}$ by deleting all non-contractible components of $\widehat{\mathcal{C}}$, deleting each edge that projects onto $xy$, and then contracting each $C_i$ to a vertex $c_i$.
		Then $G$ is bipartite with $|V(G)|=2p+r$.
		
		Each $c_i$ is adjacent to vertices of the form $x_j$ and $y_k$, and so certainly has minimum degree at least 2.
		If $c_i$ has degree exactly 2 for some $i$, then $\widehat{H}$ would contain an embedding of $K_5-e$ on vertex set $C_i\cup \{x_j,y_k\}$ where $C_i$ is a facial triangle, $x_iy_k$ is not an edge, and there is no non-contractible cycle.
		But there is no such planar embedding of $K_5 - e$, so we may conclude that every $c_i$ has minimum degree at least 3. This implies that $|E(G)|\ge 3r$.
		
		Plugging into \cref{EulerBi}, we obtain $3r \le 2(2p+r)-4+2g$, so $r\le 4p+2g-4$, and therefore $kp/2 \le 4p+2g-4 \le 4p+2g$.
		Since $k > 8$, this gives $p\le \frac{4g}{k-8} \le \frac{28g}{k-8}$, as desired.
	\end{proof}

	\begin{lemma}
		\label{lem:pi}
		Let $\Sigma$ be a surface of Euler genus $g$ and let $k\ge 25$.
		If $\Pi_{1,k}$ has a finite $p$-ply $\Sigma$-cover, then $p\le \frac{28g}{k-24}$.
	\end{lemma}
	
	\begin{proof}
		Let $\widehat{H}$ be a $p$-ply cover of $\Pi_{1,k}$ embedded in $\Sigma$.
		Let $\mathcal{C}$ be the hanging triangles of $\Pi_{1,k}$.
		By \cref{facial} and \cref{lem:finply_noncontractible}, we may assume that for every hanging triangle $C$ of $H$, every contractible component of $\widehat{C}$ is a facial triangle, and at least half of the vertices of $\widehat{\mathcal{C}}$ are contained in contractible components.
		
		Let $P$ be the joining path of $\Pi_{1,k}$.
		Then, $\widehat{P}$ consists of $p$ paths $P_1, \ldots , P_p$ isomorphic to $P$.
		Let $C_1,\ldots , C_r$ be the contractible (facial triangle) components of $\widehat{\mathcal{C}}$.
		Then $kp/2 \le r$.
		Consider the minor $G_1$ of $\widehat{H}$, obtained from $\widehat{H}$ by performing the following operations:
		\begin{enumerate}
			\item delete all the non-contractible components of $\widehat{\mathcal{C}}$;
			\item for each $1\le i \le r$, contract the component $C_i$ to a vertex $c_i$;
			\item for each $1\le i \le p$, contract the path $P_i$ to a vertex $p_i$.
		\end{enumerate}
		It is clear that $G_1$ is bipartite and $|V(G_1)| = p+r$.
		Since $K_5$ is non-planar, it follows that each $c_i$ has degree at least two in $G_1$.
		Partition $\{c_1, \ldots , c_r\}$ into two sets: $X$ consisting of those with degree at least three, and $Y$ consisting of those with degree two.
		If $|X|\ge r/2\ge kp/4$, then $|E(G_1)|\ge 2r+|X|\ge 5r/2$.
		By \cref{EulerBi}, this gives $|E(G_1)|\le 2(p+r) +2g -4$, so $kp/4 \le r/2 \le 2p +2g -4 \le 2p +2g$. Hence, we have $p\le \frac{8g}{k-8} \le \frac{28g}{k-24}$ as desired.
		
		Otherwise, we may assume that $|Y|\ge r/2\ge kp/4$.
		In this case, let $G_2$ be the minor of $G_1$ obtained by deleting the vertices of $X$, and contracting one of the two edges incident to $c$ for every $c\in Y$.
		Then $V(G_2)=\{p_1, \ldots , p_p\}$ is a set of size $p$.
		Suppose for sake of contraction that $|E(G_2)| < |Y|/2$.
		Then there exist vertices $c_1',c_2',c_3'\in Y$ and $p_1',p_2'\in \{p_1, \ldots , p_p\}$ of $G_1$ such that the three paths $p_1'c_1'p_2'$, $p_1'c_2'p_2'$, and $p_1'c_3'p_2'$ all have the same homotopy type.
		This implies that in $\widehat{H}$ there is a planar subgraph $J$ consisting of
		\begin{itemize}
			\item two paths, say $x_1\ldots x_{k+1}$ and $y_1\ldots y_{k+1}$,
			\item three triangles, say $T_1$, $T_2$, and $T_3$,
			\item for some $1\le a_1\le a_2\le a_3 \le k$, we have that for each $1\le i \le 3$, $T_i$ is contained in the neighbourhoods of both $\{x_{a_i},y_{a_i}\}$ and $\{x_{a_i+1},y_{a_i+1}\}$, and furthermore each of $x_{a_i},y_{a_i}, x_{a_i+1},y_{a_i+1}$ have one or two neighbours in $T_i$.
		\end{itemize}
		Observe that in $J$, there is a path $Q_1$ between $x_{a_2}$ and $y_{a_2}$ on the vertex set $T_1\cup \{x_1,\ldots , x_{a_2}\} \cup \{y_1,\ldots , y_{a_2}\}$, and another path $Q_3$ between $x_{a_2+1}$ and $y_{a_2+1}$ on the vertex set $T_3\cup \{x_{a_2+1},\ldots , x_{k+1}\} \cup \{y_{a_2+1},\ldots , y_{k+1}\}$.
		By contracting the paths $Q_1$ and $Q_3$, we obtain $K_5$ as a minor of $J$. This contradicts the fact that $J$ is planar.
		Thus, $|E(G_2)| \ge |Y|/2 \ge r/4 \ge kp/8$.
		By \cref{Euler} we then have $kp/8 \le 3p +3g -6 \le 3p +3g$. This gives $p\le \frac{24g}{k-24} \le \frac{28g}{k-24}$, as desired.
	\end{proof}
	
	\cref{lem:3main} follows immediately from \cref{lem:K3}, \cref{lem:omega}, \cref{lem:theta}, and \cref{lem:pi}.
	
	\section*{Acknowledgements}
	
	We thank Nathan Bowler, Jim Geelen, Paul Seymour, and Sebastian Wiederrecht for helpful discussions leading us to \cref{thm:RS} and \cref{thm:Georgakopoulos}.


\begin{thebibliography}{10}
		
		\bibitem{annor2023counterexamples}
		Dickson Y.~B. Annor, Yuri Nikolayevsky, and Michael~S Payne,
		\emph{Counterexamples to {N}egami's conjecture have ply at least 14}, arXiv
		preprint arXiv:2311.01672 (2023).
		
		\bibitem{archdeacon1981kuratowski}
		Dan Archdeacon, \emph{A {K}uratowski theorem for the projective plane}, Journal
		of Graph Theory \textbf{5} (1981), no.~3, 243--246.
		
		\bibitem{archdeacon2002two}
		Dan Archdeacon, \emph{Two graphs without planar covers}, Journal of Graph Theory
		\textbf{41} (2002), no.~4, 318--326.
		
		\bibitem{ABY63}
		L.~Auslander, T.A. Brown, and J.W.T. Youngs, \emph{The imbeddings of graphs in
			manifolds}, Journal of Mathematics and Mechanics \textbf{12} (1963), no.~4,
		629--634.
		
		\bibitem{fellows}
		Michael Fellows, \emph{Planar emulators and planar covers}, Unpublished
		manuscript (1988).
		
		\bibitem{georgakopoulos2023excluded}
		Agelos Georgakopoulos, \emph{The excluded minors for embeddability into a
			compact surface}, arXiv preprint arXiv:2301.11042 (2023).
		
		\bibitem{GHW79}
		Henry~H. Glover, John~P. Huneke, and Chin~San Wang, \emph{103 graphs that are
			irreducible for the projective plane}, Journal of Combinatorial Theory,
		Series B \textbf{27} (1979), no.~3, 332--370.
		
		\bibitem{hatcher2005algebraic}
		Allen Hatcher, \emph{Algebraic topology}, Cambridge University Press, 2000.
		
		\bibitem{hlineny1998k4}
		Petr Hlin{\v{e}}n{\`y}, \emph{{$K_{4, 4}-e$ has no finite planar cover}},
		Journal of Graph Theory \textbf{27} (1998), no.~1, 51--60.
		
		\bibitem{hlineny1999note}
		Petr Hlin{\v{e}}n{\`y}, \emph{A note on possible extensions of {N}egami's conjecture}, Journal
		of Graph Theory \textbf{32} (1999), no.~3, 234--240.
		
		\bibitem{hlineny1999planar}
		Petr Hlin{\v{e}}n{\`y}, \emph{{Planar covers of graphs: Negami's conjecture}}, Phd thesis,
		Georgia Institute of Technology, 1999, Available at
		\url{http://hdl.handle.net/1853/29449}.
		
		\bibitem{hlineny2001another}
		Petr Hlin{\v{e}}n{\`y}, \emph{Another two graphs with no planar covers}, Journal of Graph
		Theory \textbf{37} (2001), no.~4, 227--242.
		
		\bibitem{hlineny201020}
		Petr Hlin{\v{e}}n{\`y}, \emph{20 years of {N}egami’s planar cover conjecture}, Graphs and
		Combinatorics \textbf{26} (2010), 525--536.
		
		\bibitem{hlineny2004possible}
		Petr Hlin{\v{e}}n{\`y} and Robin Thomas, \emph{On possible counterexamples to
			{N}egami's planar cover conjecture}, Journal of Graph Theory \textbf{46}
		(2004), no.~3, 183--206.
		
		\bibitem{kawarabayashi2009hadwiger}
		Ken-ichi Kawarabayashi and Bruce Reed, \emph{Hadwiger's conjecture is
			decidable}, Proceedings of the forty-first annual ACM symposium on Theory of
		computing, 2009, pp.~445--454.
		
		\bibitem{kuratowski1930}
		Kazimierz Kuratowski, \emph{Sur le probl\`{e}me des courbes gauches en
			topologie}, Fundamenta Mathematicae \textbf{15} (1930), 283.
		
		\bibitem{malnivc1992generating}
		Aleksander Malni{\v{c}} and Bojan Mohar, \emph{Generating locally cyclic
			triangulations of surfaces}, Journal of Combinatorial Theory, Series B
		\textbf{56} (1992), no.~2, 147--164.
		
		\bibitem{mohar1999linear}
		Bojan Mohar, \emph{A linear time algorithm for embedding graphs in an arbitrary
			surface}, SIAM Journal on Discrete Mathematics \textbf{12} (1999), no.~1,
		6--26.
		
		\bibitem{mohar2001graphs}
		Bojan Mohar and Carsten Thomassen, \emph{Graphs on surfaces}, Johns Hopkins
		University Press, 2001.
		
		\bibitem{negami1988}
		Seiya Negami, \emph{Graphs which have no finite planar covering}, Bulletin of
		the Institute of Mathematics, Academia Sinica \textbf{16} (1988), no.~4,
		378--384.
		
		\bibitem{negami1988spherical}
		Seiya Negami, \emph{The spherical genus and virtually planar graphs}, Discrete
		mathematics \textbf{70} (1988), no.~2, 159--168.
		
		\bibitem{robertson2003graph}
		Neil Robertson and Paul Seymour, \emph{Graph minors. {X}{V}{I}. {E}xcluding a
			non-planar graph}, Journal of Combinatorial Theory, Series B \textbf{89}
		(2003), no.~1, 43--76.
		
		\bibitem{robertson2009graph}
		Neil Robertson and Paul Seymour, \emph{Graph minors. {X}{X}{I}. {G}raphs with unique linkages}, Journal
		of Combinatorial Theory, Series B \textbf{99} (2009), no.~3, 583--616.
		
		\bibitem{robertson24a}
		Neil Robertson and Paul Seymour, \emph{Excluding disjoint {K}uratowski graphs}, arXiv preprint
		arXiv:2405.05381 (2024).
		
		\bibitem{robertson24b}
		Neil Robertson and Paul Seymour, \emph{Excluding sums of {K}uratowski graphs}, arXiv preprint
		arXiv:2405.05384 (2024).
		
		\bibitem{seymour1993bound}
		Paul~D Seymour, \emph{A bound on the excluded minors for a surface}, preprint
		(1993).
		
		\bibitem{thilikos2024killing}
		Dimitrios~M Thilikos and Sebastian Wiederrecht, \emph{Killing a vortex},
		Journal of the ACM \textbf{71} (2024), no.~4, 1--56.
		
		\bibitem{wagner1937}
		Klaus Wagner, \emph{{\"{U}ber eine Eigenschaft der ebenen Komplexe}},
		Mathematische Annalen \textbf{114} (1937), 570--590.
		
	\end{thebibliography}
\end{document}